\documentclass[11pt]{amsart}
\usepackage{amsmath, amssymb}
\usepackage{amsfonts}
\usepackage{amsthm}
\usepackage{amsopn}
\usepackage{amscd}

\usepackage [all]{xy}

\newcommand{\sQ}{\mathcal{Q}}
\newcommand{\sH}{\mathcal{H}}

\newcommand{\sP}{\mathcal{P}}

\newcommand{\sV}{\mathcal{V}}

\newcommand{\comp}{\circ}

\theoremstyle{plain}
\newtheorem{thm}{Theorem}

\newtheorem{cor}[thm]{Corollary}

\newtheorem{lem}[thm]{Lemma}

\theoremstyle{definition}

\theoremstyle{remark}

\let\a\alpha
\let\b\beta
\let\g\gamma

\let\e\varepsilon

\let\ra\rightarrow

\let\ti\tilde

\begin{document}
\title{Connections and Parallel Transport}

\date{\today}
\author{Florin Dumitrescu}
\maketitle
%\tableofcontents

\begin{abstract} In this short note we give an elementary proof  of the fact that connections and their geometric parallel-transport counterpart are equivalent notions. 
\end{abstract}

\vspace{.2in}

It is well known that a connection (a.k.a. covariant derivative) on a vector bundle over a manifold gives rise to a parametrization-independent parallel transport along (piecewise) smooth paths in the manifold. A converse of this result appears in the book of Walter Poor \cite{Poor}. The definition of parallel transport used there is unnecessarily elaborate. A similar notion, called path-connection, was introduced by Mackenzie in \cite{Mac}, which characterizes connections by means of path-lifting. In this note we adopt a more natural definition of parallel transport better suited to the ``field theory" paradigm and show how parallel transport gives rise to connections. The equivalence of connections and parallel transport for line bundles is proved by Freed in the appendix B of \cite{FrCS1}. In \cite{Ba} Barrett and in a sightly different approach Caetano-Picken in \cite{CP} show that holonomy characterizes bundles with connections up to isomorphism. This result was then extended for abelian gerbes with connection by Mackaay and Picken in \cite{MP}. In \cite{SW1} Schreiber and Waldorf introduce a categorical notion of parallel transport, called transport functor, and describe bundles with connections as transport functors by an equivalence of categories. Their language of transport functors is suitable for higher notions of parallel transport, see \cite{SW2}. Another (unpublished) proof of the equivalence of connections and parallel transport  is due to Stolz. In a forthcoming paper with Stephan Stolz we will show that vector bundles with connections over a space $M$ describe 1-dimensional topological field theories over $M$. 

The applications of this theorem are quite surprising of which we will speak elsewhere (see \cite{D3}). The result is formulated for connections on principal bundles but the reader can think of vector bundles instead if she prefers.

Let $Q$ be a principal $G$-bundle over a manifold $M$, with the fiber $G$ a Lie group. 
Recall that a connection on $Q$ is a $G$-invariant distribution on $Q$ that is complementary to the canonical vertical distribution determined by the tangent spaces along the fibers.

Denote by $\sP(M)$ the space of (piecewise) smooth paths in $M$. This is an infinite dimensional smooth manifold, see for example \cite{Kli}. For a path $\g$ in $M$, let $i(\g)$ be the starting point of $\g$ and $e(\g)$ the ending point of $\g$. Consider the $G$-bundle  $QQ$ over $M\times M$, whose fiber at $(x,y)\in M\times M$ is the space of $G$-equivariant maps from $Q_x$ to $Q_y$. Consider also the pullback bundle $\sQ$ of the bundle $QQ$ over $M\times M$ via the map 
\[ \sP(M)\longrightarrow M\times M: \g\mapsto (i(\g), e(\g)). \]
We define a \textbf{parallel transport map associated to the bundle $Q$ over $M$} to be a smooth section $P$ of the bundle $\sQ$ over $\sP(M)$ that is the identity over constant paths, it is invariant under the action of the diffeomorphism group of intervals and is compatible with the diagram:

\[  \xymatrix{ \sQ\times_M\sQ \ar[r] \ar[d]
& \sQ \ar[d] \\
\sP(M)\times_M\sP(M) \ar[r] & \sP(M).}
\]
Here $\sP(M)\times_M\sP(M)$ is the space of pairs of paths $(\g_1, \g_2)$ such that $e(\g_1)=i(\g_2)$ and the horizontal lower map is the juxtaposition-of-paths map. The product $\sQ\times_M\sQ$ is defined in the expected way and the horizontal upper map is the composition map. A section $P$ of the bundle $\sQ$ is called {\it smooth} if for each family of paths parametrized by some (finite dimensional) manifold $S$, the section $P$ associates a family of smooth $G$-equivariant maps parametrized by $S$; compare \cite{Ba}, Section 2.1.1.

In other words, a parallel transport map $P$ associates smoothly (in families) to each path $\g$ in $M$ a $G$-map $P(\g):Q_{i(\g)}\ra Q_{e(\g)}$ such that:
\begin{enumerate}
\item $P(\g_x)=1_{Q_x}$, where $\g_x$ is a constant map at $x\in M$.
\item (Invariance under reparametrization) $P(\g\comp\a)=P(\g)$, where $\a$ is an orientation-preserving diffeomorphism of intervals.
\item (Compatibility under juxtaposition) $P(\g_2\star\g_1)= P(\g_2)\comp P(\g_1)$, where $\g_2\star\g_1$ is the juxtaposition of the paths $\g_1$ and $\g_2$.
\end{enumerate}
Then we prove:

\begin{thm} There is a natural 1-1 correspondence: 

\[  \left\{
\begin{array}{l}
\text{ Connections  }\\
\text{on $Q$ over $M$}
\end{array} \right\}
\leftrightarrow \left\{
\begin{array}{l}
\text{  Parallel-transport maps}\\
\text{ associated to $Q$ over $M$}
\end{array} 
\right\} \]

%\[ \left\{\text{Connections on} \ Q \right\}\leftrightarrow \left \{\text{parallel transport maps associated to }Q \text{ over }M\right \} \]

\end{thm}

\begin{proof}
It is well known how a connection gives rise to parallel transport, see for example \cite{KN1}. We only need to show how a parallel transport map $P$ associated to the bundle $Q$ over $M$ produces a connection. 

Let us begin by remarking that a parallel transport map allows us to lift paths in the manifold $M$ to paths in the bundle $Q$. More precisely, let $p\in Q$ and denote by $x$ its projection on $M$, via the bundle projection map $\pi:Q\ra M$. Let $\sP(M;0,x)$ be the space of smooth paths $\a$ in $M$ defined on a {\it fixed} (this is not a restriction on our further choices) interval $I$ containing $0$ such that $\a(0)=x$, and let $\sP(Q;0,p)$ be defined similarly. These path spaces are infinite dimensional smooth manifolds; a tangent vector to a path $\a$ in, say, $\sP(M;0,x)$ is a vector field $v$ along $\a$ in $M$, i.e. $v(t)\in T_{\a(t)}M$. Then the parallel transport map $P$ defines a smooth map 
\[ \Phi^p: \sP(M;0,x) \longrightarrow \sP(Q;0,p) \]
in an obvious way: if $\a:I\ra M$ is a path in $M$, define $\Phi^p(\a):I\ra Q$ by $ t\mapsto P(\a_{|t})(p)$, where $\a_{|t}$ is the restriction of $\a$ to the interval $[0,t]$ (or $[t,0]$ if $t<0$). In the following lemma we denote $\Phi^p$ by $\Phi$ to simplify notation, and use the dot notation for derivative.

\begin{lem}
Let $\a, \b\in \sP(M;0,x)$ such that $\dot{\a}(0)=\dot{\b}(0)=v\in T_xM$. Then $\dot{\widetilde{\Phi(\a)}}(0)=\dot{\widetilde{\Phi(\b)}}(0)$.
\end{lem}

\noindent \textit{Proof of Lemma.} Let $\a_\e:I\ra M:\ t\mapsto \a(\e t), \ 0\leq \e\leq 1$. Then $\{\a_\e\}_\e$ is a curve in $\sP(M;0,x)$ joining $c_x$, the constant path at $x\in M$, with the path $\a$. Then
\[ \frac{d}{d\e}\Big|_{\e=0} \a_\e(t)= \frac{d}{d\e}\Big|_{\e=0} \a(\e t)=t\dot{\a}(0)=tv. \]
So $\frac{d}{d\e}\big|_{\e=0} \a_\e= \{t\mapsto tv\}\in T_{c_x}\sP(M;0,x)$. Then the differential of $\Phi$ applied to this tangent vector at $c_x$ is
\[ \Phi_{*c_x}(t\mapsto tv)= \Big\{t\mapsto \frac{d}{d\e}\Big|_{\e=0} \Phi(\a_\e)(t)= \frac{d}{d\e}\Big|_{\e=0} \Phi(\a)(\e t)= t\dot{\widetilde{\Phi(\a)}}(0) \Big\} .\]
The first equality inside the braces is true since the parallel transport is \emph{invariant under reparametrization}. Similarly, 
\[  \Phi_{*c_x}(t\mapsto tv)=\{ t\mapsto t\dot{\widetilde{\Phi(\b)}}(0) \}, \]
whence the lemma.

Because of the lemma, for any tangent vector $v\in T_xM$ we can define $\ti{v}$, its lift to $p\in Q$, by
$ \ti{v}= \dot{\widetilde{\Phi(\a)}}(0)\in T_pQ$, if $v=\dot{\a}(0)$, for some $\a\in\sP(M;0,x)$. Then define 
\[ \sH_p:=\{\ti{v}\ |\ \ti{v} \text{ is the lift to $p$ of } v, \text{ for some } v\in T_xM \}. \]
Let us remark that $\sH_p$ is a linear subspace of the tangent space $T_pQ$ of $Q$ at $p$. Indeed, if we let 
\[ j:T_xM\ra T_{c_x}\sP(M;0,x): v\mapsto \{t\mapsto tv\} \]
be the inclusion of the tangent space to $M$ at $x$ into the tangent space to $\sP(M;0,x)$ at the constant path $c_x$ at $x$, then $ \Phi^p_*(j(T_xM))$ is the subspace of $T_{c_x}\sP(Q;0,p)$ consisting of tangent vectors of the form $ \{t\mapsto t\ti{v}\}$ where $\ti{v}$ is the lift of $v$, for some $ v\in T_xM$, a space which is easily identified with $\sH_p$. 

Next, let us observe that the $G$-action on the fibers of the bundle allows us to define for each group element $g\in G$ a map 
\[ R_g: \sP(Q;0,p)\ra \sP(Q;0,pg) \]
by right $g$-translation. It is then clear that
\[ \Phi^{pg}= R_g\comp\Phi^p, \]
which implies that 
\[ \sH_{pg}=R_{g*}(\sH_p). \]
In other words, the collection of spaces $\{\sH_p\}_{p\in Q}$ defines a $G$-invariant distribution on the manifold $Q$. It is also complementary to the vertical distribution on $Q$, i.e.
\[ \sH_p\oplus\sV_p= T_pQ, \text{ for all } p\in Q, \]
where $\sV_p$ is the tangent space along the fiber at $p\in Q$. This is true since $\pi_*\sH_p= T_xM$ on one side, and vertical vectors cannot be lifts of tangent vectors on the base, on the other side. The distribution $\sH=\{\sH_p\}_{p\in Q}$ therefore defines a connection on the bundle $Q$, whose parallel transport is clearly the parallel transport map we started with. The two constructions are inverses of each other and this ends the proof.

\end{proof}

A {\it homotopy-invariant} parallel transport map associated to $Q$ over $M$ is a parallel transport map $P$ which is invariant under homotopies of paths via smooth deformations with fixed endpoints. As a consequence of the theorem, we obtain the following:

\begin{cor} There is a natural 1-1 correspondence: 
\[  \left\{
\begin{array}{l}
\text{Flat connections  }\\
\text{\ \ on $Q$ over $M$}
\end{array} \right\}
\leftrightarrow \left\{
\begin{array}{l}
\text{homotopy-invariant parallel transport }\\
\text{ \ \ \  maps associated to $Q$ over $M$}
\end{array} 
\right\} \]

\end{cor}

\begin{proof} 
It is well-known that a bundle with flat connection gives rise to a parallel transport map that is homotopy-invariant. Conversely, given such a parallel-transport map $P$ associated to the bundle $Q$ over $M$, by the previous theorem, there is a connection on $Q$ over $M$  whose parallel transport is given by $P$. The curvature of such a connection is zero, since its parallel transport map is homotopy-invariant (see for example \cite{GKM}, Section 2.6).
\end{proof}

\noindent\emph{Acknowledgements.} The author would like to thank Frederico Xavier and Stephan Stolz for helpful discussions. We would also like to thank the reviewer(s) for their pertinent remarks and for signaling to us a vast preexisting literature on the subject.

\bibliographystyle{plain}
\bibliography{biblio}

\begin{thebibliography}{10}

\bibitem{Ba}
J.~W. Barrett.
\newblock Holonomy and path structures in general relativity and {Y}ang-{M}ills
  theory.
\newblock {\em Internat. J. Theoret. Phys.}, 30(9):1171--1215, 1991.

\bibitem{CP}
A.~Caetano and R.~F. Picken.
\newblock An axiomatic definition of holonomy.
\newblock {\em Internat. J. Math.}, 5(6):835--848, 1994.

\bibitem{D3}
Florin Dumitrescu.
\newblock On 2-dimensional topological field theories.
\newblock {I}n preparation, 2010.

\bibitem{FrCS1}
Daniel~S. Freed.
\newblock Classical {C}hern-{S}imons theory. {I}.
\newblock {\em Adv. Math.}, 113(2):237--303, 1995.

\bibitem{GKM}
D.~Gromoll, W.~Klingenberg, and W.~Meyer.
\newblock {\em Riemannsche {G}eometrie im {G}ro\ss en}.
\newblock Lecture Notes in Mathematics, Vol. 55. Springer-Verlag, Berlin, 1975.
\newblock Zweite Auflage.

\bibitem{Kli}
Wilhelm Klingenberg.
\newblock {\em Riemannian geometry}, volume~1 of {\em de Gruyter Studies in
  Mathematics}.
\newblock Walter de Gruyter \& Co., Berlin, 1982.

\bibitem{KN1}
Shoshichi Kobayashi and Katsumi Nomizu.
\newblock {\em Foundations of differential geometry. {V}ol. {I}}.
\newblock Wiley Classics Library. John Wiley \& Sons Inc., New York, 1996.
\newblock Reprint of the 1963 original, A Wiley-Interscience Publication.

\bibitem{MP}
Marco Mackaay and Roger Picken.
\newblock Holonomy and parallel transport for abelian gerbes.
\newblock {\em Adv. Math.}, 170(2):287--339, 2002.

\bibitem{Mac}
K.~Mackenzie.
\newblock {\em Lie groupoids and {L}ie algebroids in differential geometry},
  volume 124 of {\em London Mathematical Society Lecture Note Series}.
\newblock Cambridge University Press, Cambridge, 1987.

\bibitem{Poor}
Walter~A. Poor.
\newblock {\em Differential geometric structures}.
\newblock McGraw-Hill Book Co., New York, 1981.

\bibitem{SW2}
Urs Schreiber and Konrad Waldorf.
\newblock Connections on non-abelian gerbes and their holonomy.
\newblock http://arxiv.org/abs/0808.1923, 2008.

\bibitem{SW1}
Urs Schreiber and Konrad Waldorf.
\newblock Parallel transport and functors.
\newblock {\em J. Homotopy Relat. Struct.}, 4(1):187--244, 2009.

\end{thebibliography}

\bigskip
\raggedright Max-Planck-Institut  f\"{u}r Mathematik\\  53111 Bonn, Germany\\
 Email: {\tt florin@mpim-bonn.mpg.de}

\end{document}